\documentclass{article}
\usepackage{amssymb,bm}
\usepackage{amsmath}
\usepackage{amsthm}
\usepackage{graphicx}
\usepackage[active]{srcltx} % Includes line number info into dvi file
\usepackage{hyperref}
\hypersetup{pdfborder=0 0 0}

\newtheorem{theorem}{{\sc Theorem}}[section]

\newtheorem{lemma}[theorem]{{\sc Lemma}}
\newtheorem{corollary}[theorem]{Corollary}
\newtheorem{remark}[theorem]{Remark}

\newcommand{\bb}[1]{\mathbb{ #1}}

\newcommand{\dOm}{\partial\Omega}

\newcommand{\bra}[1]{\overline{#1}}

\newcommand{\hf}{\displaystyle\frac{1}{2}}
\newcommand{\nth}[1]{\displaystyle\frac{1}{#1}}

\newcommand{\Grad}{\nabla}

\renewcommand{\Hat}[1]{\widehat{#1}}
\newcommand{\Tld}[1]{\widetilde{#1}}

\def\XXint#1#2#3{{\setbox0=\hbox{$#1{#2#3}{\int}$ }
\vcenter{\hbox{$#2#3$ }}\kern-.6\wd0}}

\newcommand{\re}{\Re\mathfrak{e}}

\newcommand{\mat}[4]{\left[\begin{array}{cc}
\displaystyle{#1}&\displaystyle{#2}\\[1ex]
\displaystyle{#3}&\displaystyle{#4}\end{array}\right]}

\newcommand{\bc}{boundary condition}

\newcommand{\rhs}{right-hand side}

\newcommand{\mc}{microstructure}

\newcommand{\Ga}{\alpha}

\newcommand{\Ge}{\epsilon}

\newcommand{\Gth}{\theta}

\newcommand{\Gs}{\sigma}

\newcommand{\Go}{\omega}

\newcommand{\Gz}{\zeta}
\newcommand{\GD}{\Delta}

\newcommand{\GO}{\Omega}

\bmdefine\BGa{\alpha}
\bmdefine\BGb{\beta}
\bmdefine\BGd{\delta}
\bmdefine\BGe{\epsilon}
\bmdefine\BGve{\varepsilon}
\bmdefine\BGf{\phi}
\bmdefine\BGvf{\varphi}
\bmdefine\BGg{\gamma}
\bmdefine\BGc{\chi}
\bmdefine\BGi{\iota}
\bmdefine\BGk{\kappa}
\bmdefine\BGl{\lambda}
\bmdefine\BGn{\eta}
\bmdefine\BGm{\mu}
\bmdefine\BGv{\nu}
\bmdefine\BGp{\pi}
\bmdefine\BGth{\theta}
\bmdefine\BGvth{\vartheta}
\bmdefine\BGr{\rho}
\bmdefine\BGvr{\varrho}
\bmdefine\BGs{\sigma}
\bmdefine\BGvs{\varsigma}
\bmdefine\BGt{\tau}
\bmdefine\BGj{\tau}
\bmdefine\BGu{\upsilon}
\bmdefine\BGo{\omega}
\bmdefine\BGx{\xi}
\bmdefine\BGy{\psi}
\bmdefine\BGz{\zeta}
\bmdefine\BGD{\Delta}
\bmdefine\BGF{\Phi}
\bmdefine\BGG{\Gamma}
\bmdefine\BGL{\Lambda}
\bmdefine\BGP{\Pi}
\bmdefine\BGT{\Theta}
\bmdefine\BGS{\Sigma}
\bmdefine\BGU{\Upsilon}
\bmdefine\BGO{\Omega}
\bmdefine\BGX{\Xi}
\bmdefine\BGY{\Psi}

\newcommand{\CC}{{\mathcal C}}

\newcommand{\CP}{{\mathcal P}}

\bmdefine\BCA{{\mathcal A}}
\bmdefine\BCB{{\mathcal B}}
\bmdefine\BCC{{\mathcal C}}
\bmdefine\BCD{{\mathcal D}}
\bmdefine\BCE{{\mathcal E}}
\bmdefine\BCF{{\mathcal F}}
\bmdefine\BCG{{\mathcal G}}
\bmdefine\BCH{{\mathcal H}}
\bmdefine\BCI{{\mathcal I}}
\bmdefine\BCJ{{\mathcal J}}
\bmdefine\BCK{{\mathcal K}}
\bmdefine\BCL{{\mathcal L}}
\bmdefine\BCM{{\mathcal M}}
\bmdefine\BCN{{\mathcal N}}
\bmdefine\BCO{{\mathcal O}}
\bmdefine\BCP{{\mathcal P}}
\bmdefine\BCQ{{\mathcal Q}}
\bmdefine\BCR{{\mathcal R}}
\bmdefine\BCS{{\mathcal S}}
\bmdefine\BCT{{\mathcal T}}
\bmdefine\BCU{{\mathcal U}}
\bmdefine\BCV{{\mathcal V}}
\bmdefine\BCW{{\mathcal W}}
\bmdefine\BCX{{\mathcal X}}
\bmdefine\BCY{{\mathcal Y}}
\bmdefine\BCZ{{\mathcal Z}}

\bmdefine\Bzr{ 0}
\bmdefine\Ba{ a}
\bmdefine\Bb{ b}
\bmdefine\Bc{ c}
\bmdefine\Bd{ d}
\bmdefine\Be{ e}
\bmdefine\Bf{ f}
\bmdefine\Bg{ g}
\bmdefine\Bh{ h}
\bmdefine\Bi{ i}
\bmdefine\Bj{ j}
\bmdefine\Bk{ k}
\bmdefine\Bl{ l}
\bmdefine\Bm{ m}
\bmdefine\Bn{ n}
\bmdefine\Bo{ o}
\bmdefine\Bp{ p}
\bmdefine\Bq{ q}
\bmdefine\Br{ r}
\bmdefine\Bs{ s}
\bmdefine\Bt{ t}
\bmdefine\Bu{ u}
\bmdefine\Bv{ v}
\bmdefine\Bw{ w}
\bmdefine\Bx{ x}
\bmdefine\By{ y}
\bmdefine\Bz{ z}
\bmdefine\BA{ A}
\bmdefine\BB{ B}
\bmdefine\BC{ C}
\bmdefine\BD{ D}
\bmdefine\BE{ E}
\bmdefine\BF{ F}
\bmdefine\BG{ G}
\bmdefine\BH{ H}
\bmdefine\BI{ I}
\bmdefine\BJ{ J}
\bmdefine\BK{ K}
\bmdefine\BL{ L}
\bmdefine\BM{ M}
\bmdefine\BN{ N}
\bmdefine\BO{ O}
\bmdefine\BP{ P}
\bmdefine\BQ{ Q}
\bmdefine\BR{ R}
\bmdefine\BS{ S}
\bmdefine\BT{ T}
\bmdefine\BU{ U}
\bmdefine\BV{ V}
\bmdefine\BW{ W}
\bmdefine\BX{ X}
\bmdefine\BY{ Y}
\bmdefine\BZ{ Z}

\title{Exact scaling exponents in Korn and Korn-type inequalities for cylindrical shells}
\author{Yury Grabovsky \and Davit Harutyunyan}
\begin{document}
\maketitle
\begin{abstract}
  Understanding asymptotics of gradient components in relation to the
  symmetrized gradient is important for the analysis of buckling of slender
  structures.  For circular cylindrical shells we obtain the exact scaling
  exponent of the Korn constant as a function of shell's thickness. Equally
  sharp results are obtained for individual components of the gradient in
  cylindrical coordinates. We also derive an analogue of the Kirchhoff ansatz,
  whose most prominent feature is its singular dependence on the slenderness
  parameter, in marked contrast with the classical case of plates and rods.
\end{abstract}

%\tableofcontents

\section{Introduction}
\label{sec:intro}
Korn's inequalities \cite{korn08,korn09} play a central role in the theory of
linear and non-linear elasticity, and other areas of physics (see the review
\cite{horg95}). In the study of buckling of slender structures under
compression \cite{hold64,dasu06,grtr07} and in the larger study of safe loads
\cite{beat71,delp80,fffap10} of fundamental importance is the dependence of
the Korn constant on parameters of the problem.  With the
application to buckling in mind we study the scaling of the Korn constant as a
power of $h=t/R$, where $t$ is the wall thickness and $R$ is the outer radius
of the circular cylindrical shell. (See \cite{pato12,pato13} for the
application of this theory to rods and plates.) Traditionally, the Korn
inequality is proved either for the functions in the orthogonal complement to
the space of infinitesimal motions \cite{korn08,korn09,frie47} or for functions
vanishing on a portion of the boundary \cite{kool88}. However, for the study
of buckling and in other contexts one needs to examine the Korn constant for
more general spaces of functions that contain no infinitesimal motions
\cite{Kohn:1985:NMT,ryzh99,devi02,grtr07,lemu11}.

In order to obtain an asymptotically sharp estimate of the Korn constant one
needs to provide an ``ansatz'': a family of near-minimizers for the variational
definition of the Korn constant and then prove an ``ansatz-free'' inequality
establishing the sharpness of the ansatz. This program can be completed for
both linear and non-linear versions of the Korn inequality via a compactness
theorem for rods and plates \cite{fjm02}, justifying the Kirchhoff ansatz
\cite{kirch1850}. However, the compactness does not hold for cylindrical
shells and new approaches, including a new ansatz are needed. The ansatz
presented in this paper involves oscillations on a scale $h^{1/4}$,
intermediate between the macroscopic and the length scale $h$ of the shell
wall. Our method of proof of the ansatz-free bound reduces the first Korn
inequality for the circular cylindrical shell to 2D Korn-type inequalities
defined on the cylindrical coordinate ``plane'' cross-sections. These
Korn-type inequalities can be regarded to be a cross between the first and
second Korn inequalities \cite{oshyo09}. The proof uses the method of harmonic
projections from \cite{kool88}. The great flexibility of this method was
further explored in \cite{haru14} with the eventual goal of establishing
Korn's inequalities for other shells or for imperfect cylindrical shells,
needed for understanding the strong sensitivity to imperfections of the critical
buckling load of axially compressed circular cylindrical shells.

Another quantity called for by the buckling theory from \cite{grtr07} is the
Korn-like constants in the Korn-like inequalities for gradient
components. These inequalities have the form of the first Korn inequality but
with a specific component of the gradient matrix in place of the full
gradient. We show that, perhaps surprisingly, the Korn-like constants scale in
$h$ differently from the Korn constant. This phenomenon is a manifestation of
a high degree of symmetry in circular cylindrical shells. With this
understanding it is natural that our proof makes full use of that symmetry
through the periodicity and the transformation of the problem to the Fourier
space. We conjecture that the imperfections breaking the symmetry will also
destroy the distinct power laws in the Korn-like inequalities for gradient
components. We believe that this effect of imperfections is related to the large
discrepancy between the theoretical buckling load \cite{lorenz11,timosh14} and
the experimentally observed values \cite{call2000,lcp00}. These ideas are made
more precise in our companion paper \cite{grha}.

This paper is organized as follows.  In Section~\ref{sec:KI} we introduce Korn
and Korn-like constants and state our main results for the cylindrical shell. The
new oscillatory ansatz is also derived there. We reduce the ansatz-free Korn
inequality for the cylindrical shell to the 2D Korn-type inequalities in
Section~\ref{sec:red}. These inequalities are proved by means of the harmonic
projection method in Section~\ref{sec:2d}. In Section~\ref{sec:gradcomp} we
prove Korn-like inequalities for gradient components by going to the Fourier
space, and using the understanding that simple algebra in Fourier space translates into
highly non-trivial statements in the language of differential calculus.

\section{Korn and Korn-type inequalities for cylindrical shells}
\setcounter{equation}{0}
\label{sec:KI}
Let $\GO\subset\bb{R}^{3}$ be an open set.
Let $V$ be a subspace of $W^{1,2}(\GO;\mathbb R^3)$ such that
$W_0^{1,2}(\GO;\mathbb R^3)\subset V$. We recall that the Korn's constant
$K=K(V)$ is defined by
\begin{equation}
\label{def:Korn.constant}
K(V)=\sup\{K\ge 0:\|e(\BGf)\|^2\ge K\|\nabla \BGf\|^2\text{ for all }\BGf\in V\},
\end{equation}
where
\[
e(\BGf)=\hf\left(\Grad\BGf+(\Grad\BGf)^{T}\right)
\]
and $\|\cdot\|$ always denotes the $L^{2}$ norm on the domain of definition
of the function within the norm symbol. Equivalently,
\[
K(V)=\inf_{\BGf\in V}\frac{\|e(\BGf)\|^2}{\|\nabla \BGf\|^2}.
\]
In this paper we consider a family of circular cylindrical shells given in
cylindrical coordinates $(r,\Gth,z)$ as
\[
\CC_{h}=\{(r,\Gth,z):r\in I_{h},\ \Gth\in\bb{T},\ z\in [0,L]\},\qquad I_{h}=
\left[1-\frac{h}{2},1+\frac{h}{2}\right],
\]
where $\bb{T}$ is a 1-dimensional torus (circle) describing $2\pi$-periodicity in
$\Gth$.

Our goal is to examine the asymptotics of the Korn constant
$K(V_{h})$, as $h\to 0$, where $V_{h}$ is one of the subspaces
\begin{equation}
\label{BCzero}
V_{h}^1=\{\BGf\in W^{1,2}(\CC_{h};\mathbb R^3):\BGf(r,\Gth,0)=\Bzr,\
\phi_{r}(r,\Gth,L)=\phi_{\Gth}(r,\Gth,L)=0\}
\end{equation}
or
\begin{equation}
\label{Breather}
V_{h}^2=\{\BGf\in W^{1,2}(\CC_{h};\mathbb R^3):\phi_{\Gth}(r,\Gth,0)=\phi_{z}(r,\Gth,0)=
\phi_{\Gth}(r,\Gth,L)=\phi_{z}(r,\Gth,L)=0\}.
\end{equation}
This problem arises in the theory of buckling of slender bodies \cite{grtr07},
applied to circular cylindrical shells in our companion paper \cite{grha}.  In
the first case the bottom of the shell is kept fixed, while the top is allowed
only vertical displacements under the applied loads, in the second case the
loaded shell can ``breathe'', since the radial displacements are not
prescribed at either end.  In our notation the dependence on $L$ is
suppressed, while the essential dependence on $h$ is emphasized.

In cylindrical coordinates the gradient of $\BGf=\phi_r\Be_r+\phi_\theta
\Be_\theta+ \phi_z\Be_z, $has the form
\begin{equation}
\label{gradient}
\nabla \BGf=
\begin{bmatrix}
\phi_{r,r} & \dfrac{\phi_{r,\theta}-\phi_\theta}{r} & \phi_{r,z}\\
\phi_{\theta,r} & \dfrac{\phi_{\theta,\theta}+\phi_r}{r} & \phi_{\theta,z}\\
\phi_{z,r} & \dfrac{\phi_{z,\theta}}{r} & \phi_{z,z}\\
\end{bmatrix}.
\end{equation}

\subsection{Ansatz}
The famous Kirchhoff ansatz \cite{kirch1850} for columns and plates can be derived by
substituting the quadratic Taylor expansion of
a test function $\BGf(x,y,z,h)$ defined on $\Go\times[0,h]$ around
$(z,h)=(0,0)$ into $e(\BGf)$ and postulating cancellation of zeroth order
terms \cite{grtr07}. This simple and natural method for obtaining the ansatz for the
Korn inequality does not work in the case of a cylindrical shell, implying
that the dependence on $(r,h)$ is not smooth. This may manifest itself in the
formation of small scale \mc\ as $h\to 0$. We postulate that the dependence on
$r$ is, in fact, regular:
\begin{equation}
  \label{regr}
  \BGf^{h}(r,\Gth,z)=\Bu^{h}(\Gth,z)+(r-1)\Bv^{h}(\Gth,z).
\end{equation}
We substitute this ansatz into the the formula for $e(\BGf^{h})$ in cylindrical
coordinates and attempt to eliminate all terms of order zero in $r-1$. This is
possible, except for $\phi^{h}_{z,z}$:
\[
v_{r}^{h}=0,\quad v_{\Gth}^{h}=-u^{h}_{r,\Gth}+u^{h}_{\Gth},\quad
v_{z}^{h}=-u_{r,z}^{h},\quad
u_{r}^{h}=-u^{h}_{\Gth,\Gth},\quad u^{h}_{z,\Gth}=-u^{h}_{\Gth,z}.
\]
The last equation can be replaced with
\[
u^{h}_{\Gth}=w^{h}_{,\Gth},\qquad u_{z}^{h}=-w^{h}_{,z}.
\]
Hence, we obtain the ansatz that depends on a singe function $w^{h}(\Gth,z)$:
\[
\phi^{h}_{r}=-w^{h}_{,\Gth\Gth},\quad
\phi^{h}_{\Gth}=rw^{h}_{,\Gth}+(r-1)w^{h}_{,\Gth\Gth\Gth},\quad
\phi^{h}_{z}=-w^{h}_{,z}+(r-1)w^{h}_{,\Gth\Gth z}.
\]
In this case
\[
\Grad\BGf^{h}=
\begin{bmatrix}
  0 & -w^{h}_{,\Gth}-w^{h}_{,\Gth\Gth\Gth} &-w^{h}_{,\Gth\Gth z}\\[2ex]
w^{h}_{,\Gth}+w^{h}_{,\Gth\Gth\Gth} &
\dfrac{(r-1)(w^{h}_{,\Gth\Gth}+w^{h}_{,\Gth\Gth\Gth\Gth})}{r} &
rw^{h}_{,\Gth z}+(r-1)w^{h}_{,\Gth\Gth\Gth z}\\[2ex]
w^{h}_{,\Gth\Gth z} & \dfrac{-w^{h}_{,\Gth z}+(r-1)w^{h}_{,\Gth\Gth\Gth z}}{r} &
-w^{h}_{,zz}+(r-1)w^{h}_{,\Gth\Gth zz}
\end{bmatrix}
\]
\[
e(\BGf^{h})=
\begin{bmatrix}
  0 & 0 & 0\\[2ex]
0 & \dfrac{(r-1)(w^{h}_{,\Gth\Gth}+w^{h}_{,\Gth\Gth\Gth\Gth})}{r} &
\dfrac{(r-1)(w^{h}_{,\Gth z}+w^{h}_{,\Gth\Gth\Gth z})}{2r}\\[2ex]
0 & \dfrac{(r-1)(w^{h}_{,\Gth z}+w^{h}_{,\Gth\Gth\Gth z})}{2r} &
-w^{h}_{,zz}+(r-1)w^{h}_{,\Gth\Gth zz}
\end{bmatrix}
\]
We now assume that the functions $w^{h}(\Gth,z)$ exhibit
a small scale \mc:
\[
w^{h}(\Gth,z)=W\left(\frac{\Gth}{a_{h}},\frac{z-L/2}{b_{h}}\right),\quad
\Gth\in[-\pi,\pi],\ z\in[0,L],
\]
where
\[
\sqrt{h}<a_{h}\le 1,\quad h<b_{h}\le 1,\quad\lim_{h\to 0}a_{h}b_{h}=0.
\]
Here the function
$W(\eta,\Gz)$ can be any smooth compactly supported function on
 $(-1,1)^{2}$, while the function $w^{h}(\Gth,z)$ is extended as a
 $2\pi$-periodic function in $\Gth\in\bb{R}$.
 Thus, we compute
\[
|\Grad\BGf^{h}|^{2}=O\left(\max\left\{
\nth{a_{h}^{6}},\nth{a_{h}^{4}b_{h}^{2}},\frac{1}{b_{h}^{4}}\right\}\right).
\]
\[
|e(\BGf^{h})|^{2}=O\left(\max\left\{\frac{h^{2}}{a_{h}^{6}b_{h}^{2}},\frac{h^{2}}{a_{h}^{8}},
\frac{1}{b_{h}^{4}}\right\}\right).
\]
Thus,
\[
K(V_{h})\le C\min_{(a,b)\in[h,1]^{2}}\frac{\max\left\{\dfrac{h^{2}}{a^{6}b^{2}},
\dfrac{h^{2}}{a^{8}},\dfrac{1}{b^{4}}\right\}}
{\max\left\{\nth{a^{6}},\nth{a^{4}b^{2}},\dfrac{1}{b^{4}}\right\}}.
\]
It is a matter of simple analysis to show that the minimum is achieved at
$a=\sqrt[4]{h}$, $b=1$, giving $K(V_{h})\le Ch\sqrt{h}$.
Thus, we have proved the following theorem.
\begin{theorem}[Ansatz]
\label{th:KI:ansatz}
Let
\[
V_{h}^0=V_{h}^{1}\cap V_{h}^{2}=
\{\BGf\in W^{1,2}(\CC_{h};\mathbb R^3):\BGf(r,\Gth,0)=\BGf(r,\Gth,L)=\Bzr\}.
\]
Then there exist an absolute constant $C_{0}$ such that
\begin{equation}
  \label{KI:upper}
K(V_{h}^0)\leq C_{0}h\sqrt{h}.
\end{equation}
This is established via the ansatz
\begin{equation}
\begin{cases}
\label{ansatz0}
\phi^{h}_{r}(r,\Gth,z)=&-W_{,\eta\eta}\left(\frac{\Gth}{\sqrt[4]{h}},z\right)\\[2ex]
\phi^{h}_{\Gth}(r,\Gth,z)=&r\sqrt[4]{h}W_{,\eta}\left(\frac{\Gth}{\sqrt[4]{h}},z\right)+
\frac{r-1}{\sqrt[4]{h}}W_{,\eta\eta\eta}\left(\frac{\Gth}{\sqrt[4]{h}},z\right),\\[2ex]
\phi^{h}_{z}(r,\Gth,z)=&(r-1)W_{,\eta\eta z}\left(\frac{\Gth}{\sqrt[4]{h}},z\right)
-\sqrt{h}W_{,z}\left(\frac{\Gth}{\sqrt[4]{h}},z\right),
\end{cases}
\end{equation}
where the function $W(\eta,z)$ is a smooth compactly supported function on
 $(-1,1)\times(0,L)$, while the function $\BGf^{h}(\Gth,z)$ is extended as a
 $2\pi$-periodic function in $\Gth\in\bb{R}$.
  \end{theorem}
We remark that inequality (\ref{KI:upper}) holds for $V^{1}_{h}$ and
$V^{2}_{h}$, given by (\ref{BCzero}) and (\ref{Breather}), respectively, since
$V_{h}^{0}\subset V_{h}^{i}$, $i=1,2$.

\subsection{Ansatz-free lower bounds}
\begin{theorem}[Ansatz free lower bound]
\label{th:KI}
There exist a constant $C(L)$ depending only on $L$ such that
\begin{equation}
  \label{KI:lower}
 K(V_{h}^i)\ge C(L)h^{3/2},\quad i=1,2.
\end{equation}
\end{theorem}
The proof is conducted in two steps. In Section~\ref{sec:red} we reduce
inequality (\ref{KI:lower}) to the Korn-type inequalities in 2D that can be
regarded as refined versions of the 2D Korn inequality. In
Section~\ref{sec:2d} these 2D Korn-type inequalities are proved.

The intended application of these inequalities to buckling of cylindrical
shells requires that we also estimate the $L^{2}$ norms of the individual
components of the gradient matrix (\ref{gradient}) in terms of $\|e(\BGf)\|^{2}$.
We first observe that the inequalities
\[
\|(\Grad\BGf)_{rr}\|^{2}\le\|e(\BGf)\|^{2},\qquad\|(\Grad\BGf)_{\Gth\Gth}\|^{2}\le\|e(\BGf)\|^{2},
\qquad\|(\Grad\BGf)_{zz}\|^{2}\le\|e(\BGf)\|^{2}
\]
are obvious, as are the inequalities
\[
\|(\Grad\BGf)_{\Gth r}\|=2\|e(\BGf)_{r\Gth}-\hf(\Grad\BGf)_{r\Gth}\|\le
2\|e(\BGf)_{r\Gth}\|+\|(\Grad\BGf)_{r\Gth}\|\le2\|e(\BGf)\|+\|(\Grad\BGf)_{r\Gth}\|
\]
\[
\|(\Grad\BGf)_{zr}\|=2\|e(\BGf)_{rz}-\hf(\Grad\BGf)_{rz}\|\le
2\|e(\BGf)_{rz}\|+\|(\Grad\BGf)_{rz}\|\le2\|e(\BGf)\|+\|\phi_{r,z}\|
\]
\[
\|(\Grad\BGf)_{z\Gth}\|=2\|e(\BGf)_{\Gth z}-\hf(\Grad\BGf)_{\Gth z}\|\le
2\|e(\BGf)_{\Gth z}\|+\|(\Grad\BGf)_{\Gth z}\|\le2\|e(\BGf)\|+\|\phi_{\Gth,z}\|.
\]
The task is, therefore, to estimate the ratios
$\|(\Grad\BGf)_{r\Gth}\|/\|e(\BGf)\|$, $\|\phi_{r,z}\|/\|e(\BGf)\|$, and
$\|\phi_{\Gth,z}\|/\|e(\BGf)\|$.
\begin{theorem}
  \label{th:KLI}
There exists a constant
$C(L)$ depending only on $L$ such that for any $\BGf\in V_h^1\cup V_h^2$ we have
\begin{equation}
\label{KI:grad.compts}
\frac{\|(\Grad\BGf)_{r\Gth}\|^{2}}{\|e(\BGf)\|^{2}}\le
\frac{C(L)}{h\sqrt{h}},
\end{equation}
\begin{equation}
  \label{thetaz}
  \frac{\|\phi_{\Gth,z}\|^{2}}{\|e(\BGf)\|^{2}}\le\frac{C(L)}{\sqrt{h}},
\end{equation}
\begin{equation}
  \label{rz}
   \frac{\|\phi_{r,z}\|^{2}}{\|e(\BGf)\|^{2}}\le\frac{C(L)}{h}.
\end{equation}
\end{theorem}
We observe that inequality (\ref{KI:grad.compts}) is an immediate
consequences of the Korn inequality (\ref{KI:lower}). The other two
inequalities are proved in Section~\ref{sec:gradcomp}.  The remarkable feature
of inequalities (\ref{KI:grad.compts})--(\ref{rz}) is the presence of 3
distinct scaling laws for different components of the gradient.  This is a
consequence of the high degree of symmetry possessed by the circular
cylindrical shell. We conjecture that deviations from the perfect symmetry
will ``mix'' the three cylindrical components producing a single scaling
exponent determined by the Korn constant. Another important observation is
that ansatz (\ref{ansatz0}) exhibits the scaling laws given by the upper
bounds for all 3 ratios in Theorem~\ref{th:KLI}.

\section{Reduction to 2D Korn-type inequalities}
\setcounter{equation}{0}
\label{sec:red}
In this section we give the proof of Theorem~\ref{th:KI} modulo 2D Korn-type
inequalities, which constitute the technical core of our method. The argument
in this section splits naturally into a sequence of successive steps.

\noindent\textbf{Step 1.} In this step we prove that one can replace $\Grad\BGf$ and $e(\BGf)$
in Theorem~\ref{th:KI} by
\[
\BA=\left[
\begin{array}{ccc}
  \phi_{r,r} & \phi_{r,\Gth}-\phi_{\Gth} & \phi_{r,z}\\
  \phi_{\Gth,r} & \phi_{\Gth,\Gth}+\phi_{r} & \phi_{\Gth,z}\\
  \phi_{z,r} & \phi_{z,\Gth} & \phi_{z,z}
\end{array}
\right],\quad\text{and}\quad \BA_{\rm sym}=\hf(\BA+\BA^{T}),
\]
respectively. The justification is based on a simple observation that
\begin{equation}
\label{nabUandA}
\|e(\BGf)-\BA_{\rm sym}\|^2\leq\|\nabla\BGf-\BA\|^2\leq h^2\|\BA\|^2.
\end{equation}
Indeed, if we can prove that $\|\BA\|^{2}\le Ch^{-3/2}\|\BA_{\rm sym}\|^{2}$, then
\[
\|\BA\|^{2}\le Ch^{-3/2}\|\BA_{\rm sym}\|^{2}\le
Ch^{-3/2}(\|e(\BGf)\|^{2}+h^{2}\|\BA\|^2),
\]
and therefore $(1-C\sqrt{h})\|\BA\|^{2}\le Ch^{-3/2}\|e(\BGf)\|^{2}$. Thus, for
sufficiently small $h$ we also have
\[
\|\BA\|^{2}\le Ch^{-3/2}\|e(\BGf)\|^{2},
\]
concluding that
\[
\|\Grad\BGf\|^{2}\le 2\|\nabla\BGf-\BA\|^2+2\|\BA\|^{2}\le2(h^{2}+1)\|\BA\|^{2}\le
Ch^{-3/2}\|e(\BGf)\|^{2}.
\]
\noindent\textbf{Step 2.} In order to prove the Korn inequality for $\BA$
we need to estimate the quantities
\[
G_{12}^{2}=\|\phi_{\Gth,r}\|^{2}+\|\phi_{r,\Gth}-\phi_{\Gth}\|^{2},\qquad
G_{13}^{2}=\|\phi_{r,z}\|^{2}+\|\phi_{z,r}\|^{2},\qquad
G_{23}^{2}=\|\phi_{z,\Gth}\|^{2}+\|\phi_{\Gth,z}\|^{2}
\]
in terms of
\[
E_{12}^{2}=\|\phi_{\Gth,r}+\phi_{r,\Gth}-\phi_{\Gth}\|^{2},\qquad
E_{13}^{2}=\|\phi_{r,z}+\phi_{z,r}\|^{2},\qquad
E_{23}^{2}=\|\phi_{z,\Gth}+\phi_{\Gth,z}\|^{2},
\]
\[
E_{11}^{2}=G_{11}^{2}=\|\phi_{r,r}\|^{2},\qquad
E_{22}^{2}=G_{22}^{2}=\|\phi_{\Gth,\Gth}+\phi_{r}\|^{2},\qquad
E_{33}^{2}=G_{33}^{2}=\|\phi_{z,z}\|^{2}.
\]
\noindent\textbf{Estimate for $G_{23}$.} This estimate is the simplest to make. Integration by parts, using the \bc s $\phi_{\Gth}=0$ at $z=0$ and
$z=L$, common to the spaces $V_{h}^{1}$ and $V_{h}^{2}$, and the
periodicity in $\Gth$, gives
\[
|(\phi_{z,\Gth},\phi_{\Gth,z})|=|(\phi_{z,z},\phi_{\Gth,\Gth})|\le\|\phi_{z,z}\|\|\phi_{\Gth,\Gth}\|\le
E_{33}(E_{22}+\|\phi_{r}\|),
\]
where $(f,g)$ denotes the inner product of $f$ and $g$ in $L^{2}(\CC_{h})$.
It follows that
\begin{equation}
  \label{G23}
G_{23}^{2}=E_{23}^{2}-2(\phi_{z,\Gth},\phi_{\Gth,z})\le
E_{23}^{2}+E_{22}^{2}+E_{33}^{2}+2E_{33}\|\phi_{r}\|\le2\|\BA_{\rm sym}\|(\|\BA_{\rm sym}\|+\|\phi_{r}\|).
\end{equation}
\noindent\textbf{Estimate for $G_{13}$.} It would seem that the most natural way to estimate
$G_{13}$ is by the Korn inequality on the rectangle $I_{h}\times[0,L]$ \cite{Kohn:1985:NMT}:
\begin{equation}
  \label{KIrect}
\|e(\BGF)\|^{2}\ge Ch^{2}\|\Grad\BGF\|^{2},
\end{equation}
where
$\BGF(r,z)=(\phi_{r}(r,\Gth_{0},z),\phi_{z}(r,\Gth_{0},z))$ for each fixed
$\Gth_{0}$, since $\phi_{z}(r,\Gth_{0},0)=0$.
However, inequality (\ref{KIrect}) is incapable
of delivering the correct scaling law $h^{3/2}$ of the 3D Korn constant, and, hence, a
more delicate estimate is required.
\begin{theorem}[``First-and-a-half Korn inequality'']
  \label{th:basicineq:improved}
There exists a constant $C_0(L)>0$ depending only on $L$, such that, if the vector field $\BGf=(u,v)\in H^{1}(I_{h}\times[0,L];\bb{R}^{2})$ satisfies
$v(x,0)=0$, $x\in I_{h}$ in the sense of traces, then for
 any $h\in(0,1)$ and any $L>0$
 \begin{equation}
   \label{poltora}
   \|\nabla \BGf\|^{2}\le C_0(L)\|e(\BGf)\|\left(\frac{\|u\|}{h}+\|e(\BGf)|\right).
 \end{equation}
\end{theorem}
The theorem is proved in Section~\ref{sec:2d}.
We emphasize that there are no boundary conditions imposed on $u(x,y)$.
Applying Theorem~\ref{th:basicineq:improved} to the vector field $\BGF(r,z)$ for
every $\Gth_{0}$ and integrating the resulting
inequality in $\Gth_{0}$ over $[0,2\pi]$ we obtain, via the Cauchy-Schwartz
inequality for the product term
\begin{equation}
  \label{G13}
G_{13}^{2}\le C(L)\left(E_{11}^{2}+E_{13}^{2}+E_{33}^{2}+
\frac{\|\phi_{r}\|}{h}(E_{11}+E_{13}+E_{33})\right)\le
C(L)\|\BA_{\rm sym}\|\left(\|\BA_{\rm sym}\|+\frac{\|\phi_{r}\|}{h}\right).
\end{equation}
\noindent\textbf{Estimate for $G_{12}$.} The estimate for $G_{12}$ requires
an even more delicate Korn-type inequality for rectangles than the estimate for $G_{13}$.
\begin{theorem}
\label{th:hard}
Suppose that the vector field $\BGf=(u,v)\in
H^{1}(I_{h}\times[0,2\pi];\bb{R}^{2})$ satisfies $\BGf(x,0)=\BGf(x,2\pi)$ in
the sense of traces. Then
\begin{equation}
  \label{uest}
\|u\|^{2}\le\|\Be_{*}\|^{2}+2\|\BG_{*}\|\|v\|+2\|v\|^{2},
\end{equation}
where
\[
\BG_{*}=\mat{u_{x}}{u_{y}-v}{v_{x}}{v_{y}+u},\qquad
\Be_{*}=\hf(\BG_{*}+\BG_{*}^{T}).
\]
In addition, there exist absolute numerical constants $\Gs>0$ and
$C_{0}>0$, such that for any $h\in(0,\Gs)$
\begin{equation}
  \label{crazy}
\|\BG_{*}\|^{2}\le C_{0}\left(\|\Be_{*}\|^{2}+\|\Be_{*}\|\frac{\|u\|}{h}+\|v\|^{2}\right).
\end{equation}
\end{theorem}
The theorem is proved in Section~\ref{sec:2d}.

We apply inequality (\ref{crazy}) to the vector field
\begin{equation}
  \label{PHI}
  \BGF(r,\Gth)=(\phi_{r}(r,\Gth,z_{0}),\phi_{\Gth}(r,\Gth,z_{0}))
\end{equation}
for every $z_{0}\in[0,L]$. Integrating the resulting
inequality in $z_{0}$ and using the Cauchy-Schwarz inequality for the product
term, we obtain
\[
G_{12}^{2}\le C_{0}\left(\|\BA_{\rm sym}\|^{2}+\|\BA_{\rm sym}\|\frac{\|\phi_{r}\|}{h}+\|\phi_{\Gth}\|^{2}\right).
\]
We estimate via the 1D Poincar\'e inequality and (\ref{G23})
\begin{equation}
  \label{Poincare}
\|\phi_{\Gth}\|^{2}\le\frac{L^{2}}{\pi^{2}}\|\phi_{\Gth,z}\|^{2}\le\frac{L^{2}}{\pi^{2}}G_{23}^{2}
\le \frac{2L^{2}}{\pi^{2}}(\|\BA_{\rm sym}\|^{2}+\|\BA_{\rm sym}\|\|\phi_{r}\|).
\end{equation}
Thus, there exists a constant $C(L)\le C_{0}(L^{2}(\Gs+1)+1)$ such that
\begin{equation}
\label{G12}
G_{12}^{2}\le C(L)\|\BA_{\rm sym}\|\left(\|\BA_{\rm sym}\|+\frac{\|\phi_{r}\|}{h}\right).
\end{equation}
Combining inequalities (\ref{G23}), (\ref{G13}) and (\ref{G12}) we obtain the
3D Korn-type inequality
\begin{equation}
\label{KSIA}
\|\BA\|^2\leq C_1(L)\|\BA_{\rm sym}\|\left(\frac{\|\phi_r\|}{h}+\|\BA_{\rm sym}\|\right).
\end{equation}
It is now clear that in order to prove the Korn inequality (\ref{KI:lower}) we
need to estimate $\|\phi_r\|$.

\noindent\textbf{Estimate for $\|\phi_r\|$.} The estimate for
$\|\phi_r\|$ is based on inequality (\ref{uest}) in Theorem~\ref{th:hard}
applied to the vector field $\BGF$, given by (\ref{PHI}). Integrating the
resulting inequality in
$z_{0}$, and using the Cauchy-Schwarz
inequality for the product term we obtain
\[
\|\phi_{r}\|^{2}\le\|\BA_{\rm sym}\|^{2}+2\|\BA\|\|\phi_{\Gth}\|+2\|\phi_{\Gth}\|^{2}\le
\|\BA_{\rm sym}\|^{2}+\Ge^{2}\|\BA\|^{2}+\left(2+\nth{\Ge^{2}}\right)\|\phi_{\Gth}\|^{2}
\]
for any $\Ge>0$.
The small parameter $\Ge\in(0,1)$ will be chosen later to optimize the resulting
inequality.
By the ``Poincar\'e inequality'' (\ref{Poincare}) we obtain for
sufficiently small $\Ge$
\[
\|\phi_{r}\|^{2}\le\left(\frac{L^{2}}{\Ge^{2}}+1\right)\|\BA_{\rm sym}\|^{2}+\Ge^{2}\|\BA\|^{2}+
\frac{L^{2}}{\Ge^{2}}\|\BA_{\rm sym}\|\|\phi_{r}\|.
\]
Therefore,
\[
\|\phi_{r}\|^{2}\le2\left(\frac{L^{2}}{\Ge^{2}}+1\right)^{2}\|\BA_{\rm sym}\|^{2}+2\Ge^{2}\|\BA\|^{2}.
\]
Thus,
\begin{equation}
\label{bound.on.u_r}
\|\phi_{r}\|\le\sqrt{2}\left(\left(\frac{L^{2}}{\Ge^{2}}+1\right)\|\BA_{\rm sym}\|+\Ge\|\BA\|\right).
\end{equation}
Substituting this inequality into (\ref{KSIA}), we conclude that there is a
constant $C(L)$, depending only on $L$, such that
\[
\|\BA\|^{2}\le C(L)\left(\nth{h\Ge^{2}}+\frac{\Ge^{2}}{h^{2}}\right)\|\BA_{\rm sym}\|^{2}.
\]
We now choose $\Ge=h^{1/4}$ to minimize the upper bound:
\begin{equation}
\label{KornA}
\|\BA\|^{2}\le\frac{C(L)}{h\sqrt{h}}\|\BA_{\rm sym}\|^{2},
\end{equation}
which, due to Step 1 completes the proof of Theorem~\ref{th:KI}, modulo
Theorems~\ref{th:basicineq:improved} and \ref{th:hard}.
\begin{corollary}
\label{cor:KSI}
Inequality (\ref{KSIA}) remains valid if $A$ is replaced by $\nabla \BGf$, i.e.,
\begin{equation}
\label{KSI.gradient}
\|\nabla \BGf\|^2\leq C(L)\|e(\BGf)\|\left(\frac{\|\phi_r\|}{h}+\|e(\BGf)\|\right).
\end{equation}
\end{corollary}
\begin{proof}
Combining inequalities (\ref{KI:lower}) and (\ref{nabUandA}) we get

\begin{equation}
\label{e(phi)A_sym}
(1-C(L)h^{1/4})\|\BA_{\rm sym}\|\leq \|e(\BGf)\|\leq (1+C(L)h^{1/4})\|\BA_{\rm sym}\|,
\end{equation}

 which
 together with (\ref{KSIA}) implies (\ref{KSI.gradient}).
\end{proof}

\section{Korn and Korn-type inequalities in two dimensions}
\setcounter{equation}{0}
\label{sec:2d}
In this section our goal is to prove Theorems~\ref{th:basicineq:improved} and
\ref{th:hard}. We begin with an auxiliary lemma that will be essential in the
proof of both theorems.
\begin{lemma}
  \label{lem:basicineq}
Suppose that the vector field $\BGf(x,y)=(u(x,y),v(x,y))\in H^{1}(I_{h}\times[0,p];\bb{R}^{2})$ satisfies
$u(x,0)=u(x,p)$ in the sense of traces. Then there exists a constant $C_0(p)$ depending only on $p$ such that for
any $\Ga\in[-1,1]$, any $h\in(0,1)$ and any $p>0$
\begin{equation}
\label{basicineq}
\|\BG_{\Ga}\|^{2}\le C_0(L)\|\Be_{\Ga}\|\left(\frac{\|u\|}{h}+\|\Be_{\Ga}\|\right),
\end{equation}
where
\[
\BG_{\Ga}=\BG_{\Ga}(\BGf)=\mat{u_{x}}{u_{y}}{v_{x}}{v_{y}+\Ga u},\qquad
\Be_{\Ga}=\Be_{\Ga}(\BGf)=\hf(\BG_{\Ga}(\BGf)+(\BG_{\Ga}(\BGf))^{T}).
\]
\end{lemma}
We emphasize that there are no boundary conditions imposed on $v(x,y)$. If
$\Ga=0$, and $p=L$ then inequality (\ref{basicineq}) reduces to
(\ref{poltora}). However, the assumed \bc s in Lemma~\ref{lem:basicineq} and
Theorem~\ref{th:basicineq:improved} do not match. If $\Ga=1$, $p=\pi$, then the \bc s
in Lemma~\ref{lem:basicineq} and Theorem~ \ref{th:hard} are the same and
inequalities (\ref{basicineq}) and (\ref{crazy}) are similar, but not
identical. These small discrepancies will be rectified in the proof of the lemma.
\begin{proof}
Following the argument of Kondratiev and Oleinik in \cite{kool88}, one can assume, without
loss of generality, that $u$ is harmonic. Indeed, suppose $w(x,y)$ solves
\begin{equation}
  \label{harmon}
  \begin{cases}
  \GD w(x,y)=0,&(x,y)\in\GO\\
  w(x,y)=u(x,y),&(x,y)\in\dOm,
\end{cases}
\end{equation}
where $\GO=I_{h}\times[0,p]$. Then $\Grad w$ is the Helmholtz projection of
$\Grad u$ onto the space of divergence-free fields in
$L^{2}(\GO;\bb{R}^{2})$, and the following bounds hold:
\begin{lemma}
  \label{lem:main}
Let $\GO=I_{h}\times[0,p]$ and $\BGf=(u,v)\in H^{1}(\GO;\bb{R}^{2})$. If $w(x,y)$ is defined by (\ref{harmon}), then for
any $\Ga\in[-1,1]$, any $h\in(0,1)$, and any $p>0$
\begin{equation}
  \label{mainest}
  \|\Grad u-\Grad w\|\le\pi K_{0}\|\Be_{\Ga}\|,\qquad
\|u-w\|\le K_{0}h\|\Be_{\Ga}\|,\qquad
K_{0}=\frac{1}{\pi}\left(\sqrt{2}+\nth{\pi}\right).
\end{equation}
\end{lemma}
\begin{proof} Using the idea that the Laplacian can be expressed in terms of
  partial derivatives of components of the symmetrized gradient \cite{kool88}
  we compute, using the fact that $w$ is harmonic,
\[
\GD(u-w)=\GD u=(e^{\Ga}_{11}-e^{\Ga}_{22})_{x}+2(e^{\Ga}_{12})_{y}+\Ga e^{\Ga}_{11},
\]
in the sense of distributions. Here $e^{\Ga}_{ij}$ denote the components of the matrix $\Be_{\Ga}$.
Then, since $u-w\in H_{0}^{1}(\GO)$, we have
\[
\|\Grad(u-w)\|^{2}=\int_{\GO}\{(e^{\Ga}_{11}-e^{\Ga}_{22})(u-w)_{x}+2e^{\Ga}_{12}(u-w)_{y}+\Ga
e^{\Ga }_{11}(w-u)\}dxdy.
\]
By the Cauchy-Schwarz inequality we get
\[
\|\Grad(u-w)\|^{2}\le\|\Be_{\Ga}\|(\sqrt{2}\|\Grad(u-w)\|+|\Ga|\|u-w\|).
\]
By the Poincar\'e inequality
\[
\int_{I_{h}}|u-w|^{2}dx\le\frac{h^{2}}{\pi^{2}}\int_{I_{h}}|(u-w)_{y}|^{2}dx.
\]
Hence,
\[
\|u-w\|\le\frac{h}{\pi}\|\Grad(u-w)\|,
\]
and (\ref{mainest}) follows.
\end{proof}

Next we prove a Korn-like inequality for harmonic functions.
\begin{lemma}
  \label{lem:harmon}
Suppose $w\in H^{1}(I_{h}\times[0,p])$ is harmonic and satisfies
$w(x,0)=w(x,p)$ in the sense of traces. Then
\begin{equation}
  \label{hi}
\|w_{y}\|^{2}\le\frac{2\sqrt{3}}{h}\|w\|\|w_{x}\|+\|w_{x}\|^{2}.
\end{equation}
\end{lemma}
\begin{proof}
By the method of separation of variables
\[
w(x,y)=\sum_{n\in\mathbb Z}(A_{n}e^{\frac{2\pi nx}{p}}+B_{n}e^{-\frac{2\pi nx}{p}})
e^{\frac{2\pi n yi}{p}}
\]
in $H^{1}(I_{h}\times[0,p])$.
Therefore,
\[
\|w\|^{2}=ph\sum_{n\in\mathbb Z}\left\{\psi\left(\frac{2\pi nh}{p}\right)
\left(|A_{n}|^{2}e^{\frac{2\pi nh}{p}}+|B_{n}|^{2}e^{\frac{-2\pi nh}{p}}\right)
+2\re(A_{n}\bra{B_{n}})\right\},\quad\psi(x)=\frac{\sinh(x)}{x}.
\]
In the expansions of $w_{x}$ and $w_{y}$ we simply replace $A_{n}$ and $B_{n}$
with $2\pi n A_{n}/p$, $-2\pi n B_{n}/p$ and $2\pi in A_{n}/p$, $2\pi in
B_{n}/p$, respectively:
\[
\|w_{x}\|^{2}=4ph\sum_{n\in\mathbb Z}\frac{\pi^{2}n^{2}}{p^{2}}
\left\{\psi\left(\frac{2\pi nh}{p}\right)
\left(|A_{n}|^{2}e^{\frac{2\pi nh}{p}}+|B_{n}|^{2}e^{\frac{-2\pi nh}{p}}\right)
-2\re(A_{n}\bra{B_{n}})\right\},
\]
\[
\|w_{y}\|^{2}=4ph\sum_{n\in\mathbb Z}\frac{\pi^{2}n^{2}}{p^{2}}
\left\{\psi\left(\frac{2\pi nh}{p}\right)
\left(|A_{n}|^{2}e^{\frac{2\pi nh}{p}}+|B_{n}|^{2}e^{\frac{-2\pi nh}{p}}\right)
+2\re(A_{n}\bra{B_{n}})\right\},
\]
Denoting
\[
a_{n}=A_{n}e^{\frac{\pi nh}{p}},\qquad b_{n}=B_{n}e^{-\frac{\pi nh}{p}},\qquad\tau_{n}=\frac{2\pi nh}{p}
\]
we simplify the above expressions:
\[
\frac{\|w\|^{2}}{h^{2}}=4ph\sum_{n\in\mathbb Z}\frac{\pi^{2}n^{2}}{\tau_{n}^{2}p^{2}}
\{(\psi(\tau_{n})-1)(|a_{n}|^{2}+|b_{n}|^{2})+|a_{n}+b_{n}|^{2}\},
\]
\[
\|w_{y}\|^{2}=4ph\sum_{n\in\mathbb Z}\frac{\pi^{2}n^{2}}{p^{2}}\{(\psi(\tau_{n})-1)
(|a_{n}|^{2}+|b_{n}^{2}|)+|a_{n}+b_{n}|^{2}\},
\]
\[
\|w_{x}\|^{2}=4ph \sum_{n\in\mathbb Z}\frac{\pi^{2}n^{2}}{p^{2}}\{(\psi(\tau_{n})-1)
(|a_{n}|^{2}+|b_{n}|^{2})+|a_{n}-b_{n}|^{2}\},
\]
Obviously,
\[
\|w_{y}\|^{2}-\|w_{x}\|^{2}=16ph\sum_{n\in\mathbb Z}\frac{\pi^{2}n^{2}}{p^{2}}
\re(a_{n}\bra{b_{n}})\le16ph\sum_{n\in\CP}\frac{\pi^{2}n^{2}}{p^{2}}\re(a_{n}\bra{b_{n}}),
\]
where $\CP=\{n\in\bb{Z}:\re(a_{n}\bra{b_{n}})>0\}$.
Next we estimate
\begin{align*}
\frac{\|w\|^{2}}{h^{2}}&\ge4ph\sum_{n\in\CP}\frac{\pi^{2}n^{2}}{\tau_{n}^{2}p^{2}}
\{(\psi(\tau_{n})-1)(|a_{n}|^{2}+|b_{n}|^{2})+|a_{n}+b_{n}|^{2}\}\\
&\ge 8ph\sum_{n\in\CP}\frac{\pi^{2}n^{2}}{\tau_{n}^{2}p^{2}}(\psi(\tau_{n})+1)\re(a_{n}\bra{b_{n}}).
\end{align*}
Similarly,
\[
\|w_{x}\|^{2}\ge 8ph\sum_{n\in\CP}\frac{\pi^{2}n^{2}}{p^{2}}
(\psi(\tau_{n})-1)\re(a_{n}\bra{b_{n}}).
\]
Now we have
\[
\sum_{n\in\CP}\frac{\pi^{2}n^{2}}{p^{2}}\re(a_{n}\bra{b_{n}})=
\sum_{n\in\CP}\left(\frac{\pi n}{p}\sqrt{(\psi(\tau_{n})-1)\re(a_{n}\bra{b_{n}})}\right)
\left(\frac{\pi n}{p}\sqrt{\frac{\re(a_{n}\bra{b_{n}})}{\psi(\tau_{n})-1}}\right).
\]
Applying the Cauchy-Schwarz inequality we obtain
\[
  \sum_{n\in\CP}\frac{\pi^{2}n^{2}}{p^{2}}\re(a_{n}\bra{b_{n}})\le
\sqrt{\sum_{n\in\CP}\frac{\pi^{2}n^{2}}{p^{2}}
(\psi(\tau_{n})-1)\re(a_{n}\bra{b_{n}})}
\sqrt{\sum_{n\in\CP}\Psi(\tau_{n})
\frac{\pi^{2}n^{2}}{\tau_{n}^{2}p^{2}}(\psi(\tau_{n})+1)\re(a_{n}\bra{b_{n}})},
\]
where
\[
\Psi(\tau)=\frac{\tau^{2}}{\psi(\tau)^{2}-1}=\frac{\tau^{4}}{\sinh^{2}(\tau)-\tau^{2}}.
\]
The function $\Psi(\tau)$ is monotone decreasing on $(0,+\infty)$, and hence,
\[\Psi(\tau_{n})\le\Psi(\tau_{1})\le\Psi(0)=3.\] Therefore,
\begin{equation}
  \label{shi}
\|w_{y}\|^{2}-\|w_{x}\|^{2}\le\frac{2\sqrt{\Psi(\tau_{1})}}{h}\|w\|\|w_{x}\|,
\end{equation}
and inequality (\ref{hi}) follows.
\end{proof}
We remark that inequality (\ref{hi}) is sharp, since
\[
w(x,y)=\cosh\left(\frac{\pi}{p}\left(x-\frac{h}{2}\right)\right)
\sin\left(\frac{\pi y}{p}\right)
\]
turns the inequality (\ref{shi}) into equality.

We can now finish the proof of Lemma~\ref{lem:basicineq}.
By the triangle inequality and Lemma~\ref{lem:main} we get
\begin{multline*}
  \|\BG_{\Ga}\|^{2}=\|\Be_{\Ga}\|^{2}+\hf\|v_{x}-\phi_{y}\|^{2}=\|\Be_{\Ga}\|^{2}+
\hf\|(v_{x}+\phi_{y})-2(\phi_{y}-w_{y})-2w_{y}\|^{2}\le\\
\|\Be_{\Ga}\|^{2}+\frac{3}{2}\|\phi_{y}+v_{x}\|^{2}+6\|\phi_{y}-w_{y}\|^{2}+6\|w_{y}\|^{2}\le
(4+6\pi^{2}K_{0}^{2})\|\Be_{\Ga}\|^{2}+6\|w_{y}\|^{2}.
\end{multline*}
We estimate $\|w_{y}\|$ by means of Lemma~\ref{lem:harmon}.
By the triangle inequality and Lemma~\ref{lem:main}
\[
\|w\|\le\|u\|+\|u-w\|\le\|u\|+K_{0}h\|\Be_{\Ga}\|,
\]
\[
\|w_{x}\|\le\|\phi_{x}\|+\|w_{x}-\phi_{x}\|\le(1+\pi K_{0})\|\Be_{\Ga}\|.
\]
Therefore,
\[
\|w_{y}\|^{2}\le\frac{2\sqrt{3}(1+\pi K_{0})}{h}\|u\|\|\Be_{\Ga}\|+
(1+\pi K_{0})(1+(2\sqrt{3}+\pi)K_{0})\|\Be_{\Ga}\|^{2}.
\]
Thus, using somewhat arbitrary integer overestimation, we obtain
\begin{equation}
  \label{zbc}
\|\BG_{\Ga}\|^{2}\le100\|\Be_{\Ga}\|\left(\frac{\|u\|}{h}+\|\Be_{\Ga}\|\right).
\end{equation}
The proof of Lemma~\ref{lem:basicineq} is complete.
\end{proof}

\textbf{Proof of Theorem~\ref{th:basicineq:improved}.}
Theorem~\ref{th:basicineq:improved} follows from Lemma~\ref{lem:basicineq} via
the even-odd extension method, whereby we define the new displacement $\Tld{\BGf}=(\tilde{u},\tilde{v},)$
on the rectangle $I_{h}\times[-p, p],$ where $\tilde{u}$ and $\tilde{v}$ are extensions of $u$ and $v$ such that
$$
\tilde{u}(x,y)=
\begin{cases}
u(x,y) &\ \text{if} \ \ y\in[0, p]\\
u(x,-y) &\ \text{if} \ \ y\in[-p, 0]\\
\end{cases}\qquad
\tilde{v}(x,y)=
\begin{cases}
v(x,y) &\ \text{if} \ \ y\in[0, p]\\
-v(x,-y) &\ \text{if} \ \ y\in[-p, 0]\\
\end{cases}.
$$
We observe that due to the \bc\ $v(x,0)=0$, the extension $\Tld{\BGf}$ is an
$H^{1}(I_{h}\times[-p,p])$ vector field, while $\tilde{u}(x,-p)=\tilde{u}(x,p)$. Moreover,
$$
\nabla\Tld{\BGf}(x,y)=
\begin{cases}
\begin{bmatrix}
u_x(x,y) & u_y(x,y)\\
v_x(x,y) & v_y(x,y)
\end{bmatrix} & \ \text{if}\ \  y\in[0, p],\\[3ex]
\begin{bmatrix}
u_x(x,-y) & -u_y(x,-y)\\
-v_x(x,-y) & v_y(x,-y)
\end{bmatrix} & \ \text{if} \ \ y\in[-p, 0].\\
\end{cases}
$$
Therefore, setting $\Tld{\Omega}=I_{h}\times[-p, p]$ we get,

$$\|\nabla\Tld{\BGf}\|_{L^2(\Tld{\Omega})}^2=2\|\nabla\BGf\|_{L^2(\Omega)}^2,\qquad
\|e(\Tld{\BGf})\|_{L^2(\Tld{\Omega})}^2=2\|e(\BGf)\|_{L^2(\Omega)}^2.$$
It is also clear that
$\|\Tld{u}\|_{L^2(\Tld{\Omega})}^2=2\|u\|_{L^2(\Omega)}^2$. An
application of
Lemma~\ref{lem:basicineq} to the vector field $\Tld{\BGf}$ in the domain
$\Tld{\Omega}$ completes the proof.

\textbf{Proof of Theorem~\ref{th:hard}.} Let
$\Tld{\BGf}(x,y)=(u(x,y),(1-x)v(x,y))$, and let
\[
\Tld{\BG}=\BG_{\Ga}(\Tld{\BGf})|_{\Ga=1},\qquad\Tld{\Be}=\hf(\Tld{\BG}+\Tld{\BG}^{T}).
\]
We compute
\[
\BG_{*}=\Tld{\BG}+\mat{0}{-v}{v+xv_{x}}{xv_{y}},\qquad
\Tld{\Be}=\Be_{*}+\mat{0}{-\dfrac{x}{2}v_{x}}{-\dfrac{x}{2}v_{x}}{-xv_{y}}.
\]
Thus we immediately obtain that
\begin{equation}
  \label{Gstar}
  \|\BG_{*}\|^{2}\le6(\|\Tld{\BG}\|^{2}+\|v\|^{2}+h^{2}(\|v_{x}\|^{2}+\|v_{y}\|^{2}).
\end{equation}
and
\begin{equation}
  \label{eest}
\|\Tld{\Be}\|\le\|\Be_{*}\|+h(\|v_{x}\|+\|v_{y}\|).
\end{equation}
We also estimate
\begin{equation}
  \label{gradv}
\|v_{x}\|\le\|\BG_{*}\|,\qquad\|v_{y}\|\le\|v_{y}+u\|+\|u\|\le\|\Be_{*}\|+\|u\|.
\end{equation}
Now we apply Lemma~\ref{lem:basicineq} to the vector field $\Tld{\BGf}$ and
$\Ga=1$, and obtain
\[
\|\Tld{\BG}\|^{2}\le
C_{0}\|\Tld{\Be}\|\left(\frac{\|u\|}{h}+\|\Tld{\Be}\|\right).
\]
Therefore, by (\ref{Gstar}) and (\ref{eest}) we obtain
\[
\|\BG_{*}\|^{2}\le C_{0}\left(\|\Be_{*}\|^{2}+\|\Be_{*}\|\frac{\|u\|}{h}+
\|u\|(\|v_{x}\|+\|v_{y}\|)+\|v\|^{2}+h^{2}(\|v_{x}\|^{2}+\|v_{y}\|^{2})\right).
\]
Applying inequalities (\ref{gradv}) to the terms containing $\|v_{x}\|$
and $\|v_{y}\|$ we obtain
\[
\|\BG_{*}\|^{2}\le C_{0}\left(\|\Be_{*}\|^{2}+\|\Be_{*}\|\frac{\|u\|}{h}+
\|u\|\|\BG_{*}\|+\|u\|^{2}+\|v\|^{2}+h^{2}\|\BG_{*}\|^{2}\right).
\]
When $h^{2}<1/(2C_{0})$ we get the inequality
\[
\|\BG_{*}\|^{2}\le C_{0}\left(\|\Be_{*}\|^{2}+\|\Be_{*}\|\frac{\|u\|}{h}+
\|u\|\|\BG_{*}\|+\|u\|^{2}+\|v\|^{2}\right).
\]
We also have
\[
C_{0}\|u\|\|\BG_{*}\|\le\hf\|\BG_{*}\|^{2}+\frac{C_{0}^{2}}{2}\|u\|^{2}.
\]
Thus, we obtain
\begin{equation}
  \label{penult}
\|\BG_{*}\|^{2}\le C_{0}\left(\|\Be_{*}\|^{2}+\|\Be_{*}\|\frac{\|u\|}{h}+\|u\|^{2}+\|v\|^{2}\right).
\end{equation}
To finish the proof of the theorem we write $\|u\|^{2}$ using integration by
parts and periodic \bc s:
\[
\|u\|^{2}=(u,u+v_{y})+(u_{y}-v,v)+\|v\|^{2}.
\]
Thus,
\[
\|u\|^{2}\le\|u\|\|\Be_{*}\|+\|\BG_{*}\|\|v\|+\|v\|^{2},
\]
and using $2\|u\|\|\Be_{*}\|\le\|u\|^{2}+\|\Be_{*}\|^{2}$ we obtain (\ref{uest}).
Applying this inequality to the $\|u\|^{2}$ term in (\ref{penult}) we obtain
\[
\|\BG_{*}\|^{2}\le C_{0}\left(\|\Be_{*}\|^{2}+\|\Be_{*}\|\frac{\|u\|}{h}+\|\BG_{*}\|\|v\|+
\|v\|^{2}\right).
\]
from which Theorem~\ref{th:hard} follows.
\begin{remark}
\label{rem:imp}
  In the proofs of all of the Korn and Korn-like inequalities, the
  vanishing of $\phi_{z}(r,\Gth,L)$ was never used. Hence,
  \begin{equation}
    \label{KIplus}
    c(L)h^{3/2}\le K(V^{*}_{h})\le C(L)h^{3/2},
  \end{equation}
where
\[
V^{*}_{h}=\{\BGf\in W^{1,2}(\CC_{h};\mathbb R^3):\phi_{\Gth}(r,\Gth,0)=\phi_{z}(r,\Gth,0)=
\phi_{\Gth}(r,\Gth,L)=0\}.
\]
\end{remark}

\section{Korn inequality for gradient components}
\setcounter{equation}{0}
\label{sec:gradcomp}
The goal in this section is to prove Korn-like inequalities (\ref{thetaz})
and (\ref{rz}) for gradient components. While inequalities
(\ref{KI:grad.compts})--(\ref{rz}) bear a formal resemblance to the Korn
inequality (\ref{KI:lower}), the distinct scaling exponents in
(\ref{thetaz})--(\ref{rz}) are a consequence of the high degree of metric
symmetry in the structure. By contrast, our methods in Sections~\ref{sec:red} and
\ref{sec:2d} exploited only the topological and smooth structures of the
cylindrical shell. Not surprisingly, then, the proof of (\ref{thetaz}) and
(\ref{rz}) is based on exact calculations in Fourier space, rather than on
various integral inequalities, as in the proof of (\ref{KI:lower}). In fact,
the natural periodicity in $\Gth$ is not sufficient, and we need the
periodicity in $z$ variable as well. The boundary conditions in $V_{h}^{1}$
and $V_{h}^{2}$ permit us to achieve this goal in the same way as was done in
proof of Theorem~\ref{th:basicineq:improved} in Section~\ref{sec:2d}.  For
$V_{h}^{1}$ we extend $\phi_{r}$ and $\phi_{\Gth}$ as odd and $\phi_{z}$ as an
even function in $z\in[-L,L]$, while for $V_{h}^{2}$ we extend $\phi_{r}$ and
$\phi_{\Gth}$ as even functions as $\phi_{z}$ as odd. We remark that the
periodic extension method cannot be applied to the \bc s in the definition of
space $V_{h}^{*}$. To fix ideas we conduct the proof for the space
$V_{h}^{1}$. The proof for $V_{h}^{2}$ is obtained by switching the sine and
cosine series in the $z$ variable. Denoting the periodic extensions without
relabeling, we expand the vector field $\BGf(r,\Gth,z)$ in Fourier series in $(\Gth,z)$:
  \begin{equation}
    \label{Fourier}
\BGf(r,\Gth,z)=\sum_{m=0}^{\infty}\sum_{n\in\mathbb Z}\BGf^{(m,n)}(r,\Gth,z),
  \end{equation}
where
\[
\begin{cases}
  \phi_{r}^{(m,n)}=
\Hat{\phi}_{r}(r;m,n)\sin\left(\dfrac{\pi mz}{L}\right)e^{in\Gth},&
\Hat{\phi}_{r}(r;m,n)=\nth{\pi L}\int_{0}^{2\pi}\int_{0}^{L}
\phi_{r}\sin\left(\dfrac{\pi mz}{L}\right)e^{in\Gth}dzd\Gth\\[3ex]
\phi_{\Gth}^{(m,n)}=
\Hat{\phi}_{\theta}(r;m,n)\sin\left(\dfrac{\pi mz}{L}\right)e^{in\Gth},&
\Hat{\phi}_{\theta}(r;m,n)=\nth{\pi L}\int_{0}^{2\pi}\int_{0}^{L}
\phi_{\Gth}\sin\left(\dfrac{\pi mz}{L}\right)e^{in\Gth}dzd\Gth\\[3ex]
\phi_{z}^{(m,n)}=
\Hat{\phi}_{z}(r;m,n)\cos\left(\dfrac{\pi mz}{L}\right)e^{in\Gth},&
\Hat{\phi}_{z}(r;m,n)=\nth{\pi L}\int_{0}^{2\pi}\int_{0}^{L}
\phi_{z}\cos\left(\dfrac{\pi mz}{L}\right)e^{in\Gth}dzd\Gth.
\end{cases}
\]
We observe that in cylindrical coordinates
\[
\Grad\BGf(r,\Gth,-z)=-\left[
  \begin{array}{rrr}
    -\phi_{r,r}(r,\theta,z) & -\frac{\phi_{r,\theta}(r,\theta,z)-\phi_{\theta}(r,\theta,z)}{r} & \phi_{r,z}(r,\theta,z)\\
        -\phi_{\theta,r}(r,\theta,z) & -\frac{\phi_{\theta,\theta}(r,\theta,z)+\phi_{r}(r,\theta,z)}{r} & \phi_{\theta,z}(r,\theta,z)\\
    \phi_{z,r}(r,\theta,z) & \frac{\phi_{z,\theta}(r,\theta,z)}{r} & -\phi_{z,z}(r,\theta,z)\\
\end{array}
\right]
\]
Therefore, it is sufficient to prove inequalities (\ref{thetaz}) and
(\ref{rz}) for functions of the form
\[
\BGf^{(m,n)}(r,\Gth,z)=\left(f_{r}(r)\sin\left(\dfrac{\pi mz}{L}\right),
f_{\Gth}(r)\sin\left(\dfrac{\pi mz}{L}\right),f_{z}(r)\cos\left(\dfrac{\pi mz}{L}\right)
\right)e^{in\Gth}.
\]
Indeed,
\[
\|\phi_{r,z}\|^{2}=\pi L\sum_{m=1}^{\infty}\sum_{n\in\bb{Z}}\|\phi^{(m,n)}_{r,z}\|^{2}\le
\pi L\sum_{m=0}^{\infty}\sum_{n\in\bb{Z}}\frac{C(L)}{h}\|e(\BGf^{(m,n)})\|^{2}=
\frac{C(L)}{h}\|e(\BGf)\|^{2},
\]
with the similar bound for $\|\phi_{\Gth,z}\|$.
Observe that $\BGf^{(m,n)}\in V_{h}^{1}$ or $V_{h}^{2}$, provided $\BGf\in
V_{h}^{1}$ or $V_{h}^{2}$, respectively. Therefore, Theorem~\ref{th:KI} and
Corollary~\ref{cor:KSI} are applicable to such functions. We now fix $m\ge 1$ and
$n\in\bb{Z}$, and for simplicity of notation we use $\BGf=(\phi_{r},\phi_{\Gth},\phi_{z})$
instead of $\BGf^{(m,n)}=(\phi_{r}^{(m,n)},\phi_{\Gth}^{(m,n)},\phi_{z}^{(m,n)})$.
Notice that if $\|\phi_{r}\|\le 3\|e(\BGf)\|$, then Corollary~\ref{cor:KSI}
 implies that
\[
\|\phi_{r,z}\|^{2}\le \|\Grad\BGf\|^{2}\le\frac{C(L)}{h}\|e(\BGf)\|^{2},
\]
and (\ref{rz}) is proved. Let us now prove inequality (\ref{rz}) under the
assumption that $\|\phi_{r}\|>3\|e(\BGf)\|$. In that case inequalities
(\ref{G23}) and (\ref{KSI.gradient}) become
\begin{equation}
  \label{rtheta1}
\|\phi_{z,\Gth}\|^{2}+\|\phi_{\Gth,z}\|^{2}\le \frac{8}{3}\|e(\BGf)\|\|\phi_{r}\|
\end{equation}
and
\begin{equation}
  \label{preK1}
  \|\Grad\BGf\|^{2}\le \frac{C(L)}{h}\|e(\BGf)\|\|\phi_{r}\|,
\end{equation}
respectively.
We estimate
\[
n^2\|\phi_r\|^2=\|\phi_{r,\theta}\|^2\leq 2\|\phi_{r,\theta}-\phi_\theta\|^2+2\|\phi_\theta\|^2
\leq 2\|\nabla \BGf\|^2+\frac{2L^{2}}{\pi^{2}}\|\phi_{\theta,z}\|^2\le
C(L)\|\nabla \BGf\|^2,
\]
where the Poincar\'e inequality has been used for $\phi_{\Gth}$.
Applying inequality (\ref{preK1}) we obtain
\[
n^2\|\phi_r\|^{2}\leq C(L)\|\nabla\BGf\|^2\le\frac{C(L)}{h}\|e(\BGf)\|\|\phi_r\|.
\]
Thus,
\begin{equation}
\label{n^2urleqe(u)}
n^2\|\phi_r\|\le\frac{C(L)}{h}\|e(\BGf)\|.
\end{equation}
We next estimate
\[
\|\phi_r\|^2\leq 2\|\phi_r+\phi_{\theta,\theta}\|^2+2\|\phi_{\theta,\theta}\|^2
\leq 2\|e(\BGf)\|^2+2n^2\|\phi_\theta\|^2,
\]
and
\[
\frac{m^2\pi^2}{L^2}\|\phi_\theta\|^2=\|\phi_{\theta,z}\|^2\leq
\frac{8}{3}\|e(\BGf)\|\|\phi_r\|,
\]
due to (\ref{rtheta1}). Combining the last two inequalities we obtain
\begin{equation}
  \label{vrineq1}
\|\phi_r\|^2\leq 2\|e(\BGf)\|^2+\frac{16L^{2}n^2}{3m^2\pi^2}\|e(\BGf)\|\|\phi_r\|.
\end{equation}
By our assumption $\|e(\BGf)\|^{2}<\|\phi_{r}\|^{2}/9$. We use this
inequality to estimate the first term on the \rhs\ of (\ref{vrineq1}) and obtain
\begin{equation}
\label{u_rleqe(u)2}
\|\phi_r\|\leq\frac{12L^{2}n^2}{m^2\pi^2}\|e(\BGf)\|.
\end{equation}
Finally, multiplying (\ref{n^2urleqe(u)}) and (\ref{u_rleqe(u)2}) we get
\[
m^2\|\phi_r\|^2\leq \frac{C(L)}{h}\|e(\BGf)\|^2,
\]
and (\ref{rz}) is proved. To prove (\ref{thetaz}) we utilize (\ref{G23}) to get,
\begin{equation}
\label{thetaz<symr}
\|\phi_{\Gth,z}\|^2\leq G_{12}^{2}\leq 2\|\BA_{\rm sym}\|(\|\BA_{\rm sym}\|+\|\phi_r\|).
\end{equation}
Choosing $\Ge=\sqrt[4]{h}$ in (\ref{bound.on.u_r}) and applying (\ref{KornA})
to the resulting inequality, we obtain
$$\|\phi_r\|\leq \frac{C(L)\|\BA_{\rm sym}\|}{\sqrt h}.$$
Substituting now the last inequality into (\ref{thetaz<symr}) we get
$$\|\phi_{\Gth,z}\|^\leq \frac{C(L)}{\sqrt h}\|\BA_{\rm sym}\|^2.$$
Invoking inequality (\ref{e(phi)A_sym}), gives
$$\|\phi_{\Gth,z}\|^2\leq \frac{C(L)}{\sqrt h}\|e(\BGf)\|^2,$$
for sufficiently small $h$.
This completes the proof for the case $\BGf\in V_h^1$. If $\BGf\in
V_h^2$ we repeat the same proof changing sines to cosines in the expansion (\ref{Fourier}).

\medskip

\noindent\textbf{Acknowledgments.}
This material is based upon work supported by the National Science
Foundation under Grants No. 1008092.

% BIBLIOGRAPHY
\bibliographystyle{abbrv}
\bibliography{refs}

\end{document}